\documentclass{amsart}
\usepackage{amscd,amsfonts,amsmath,amssymb,amsthm,curves,enumerate,graphicx,latexsym,curves,multibox,hyperref}
\usepackage{bbm}
\usepackage{multirow}
\usepackage{mathrsfs}
\usepackage[mathscr]{eucal}
\usepackage{epsfig,epsf,xypic,epic}
\usepackage{epstopdf}

\def\bZ{\mathbb{Z}}

\def\bR{\mathbb{R}}
\def\bC{\mathbb{C}}
\def\bP{\mathbb{P}}

\newtheorem{thm}{Theorem}[section]
\newtheorem{lem}{Lemma}[section]
\newtheorem{cor}{Corollary}[section]
\newtheorem{prop}{Proposition}[section]
\newtheorem{defn}{Definition}[section]
\newtheorem{nb}{Remark}[section]

\numberwithin{equation}{section}

\begin{document}

\title[HMS for $A_n$-resolutions as a $T$-duality]{Homological mirror symmetry\\ for $A_n$-resolutions as a $T$-duality}
\author[K. Chan]{Kwokwai Chan}
\address{Department of Mathematics\\ The Chinese University of Hong Kong\\ Shatin \\ Hong Kong}
\email{kwchan@math.cuhk.edu.hk}

\begin{abstract}
We study Homological Mirror Symmetry (HMS) for $A_n$-resolutions from the SYZ viewpoint. Let $X\to\bC^2/\bZ_{n+1}$ be the crepant resolution of the $A_n$-singularity. The mirror of $X$ is given by a smoothing $\check{X}$ of $\bC^2/\bZ_{n+1}$. Using SYZ transformations, we construct a geometric functor from a derived Fukaya category of $\check{X}$ to the derived category of coherent sheaves on $X$. We show that this is an equivalence of triangulated categories onto a full triangulated subcategory of $D^b(X)$, thus realizing Kontsevich's HMS conjecture by SYZ.
\end{abstract}

\maketitle

\tableofcontents

\section{Introduction}

In his 1994 ICM address \cite{K94}, Kontsevich proposed a beautiful categorical framework to understand mirror symmetry. He conjectured that for a mirror pair of Calabi-Yau manifolds $Y$ and $\check{Y}$, there exists an equivalence of triangulated categories between the derived Fukaya category $D^b\textrm{Fuk}(Y)$ of $Y$ and the derived category of coherent sheaves $D^b(\check{Y})$ of $\check{Y}$ (and vice versa). This is Kontsevich's famous \emph{Homological Mirror Symmetry (HMS) conjecture}. This was later extended, again by Kontsevich \cite{K98}, to the non-Calabi-Yau setting as well.

The HMS conjecture has been verified for elliptic curves \cite{PZ98}, abelian varieties \cite{F02,KS00}, quartic K3 surfaces \cite{S03}, Calabi-Yau hypersurfaces in projective spaces \cite{S11}, del Pezzo surfaces \cite{AKO06} and projective toric manifolds \cite{S01,U06,AKO08,Ab09,F08,FLTZ08}. However, the verification was usually done by computing both sides and then either equating them to a third category (e.g. category of quiver representations) or applying deformation arguments \cite{S02}. In particular, there is often no explicit construction of the functor implementing the HMS equivalence.

In 1996, Strominger, Yau and Zaslow \cite{SYZ96} proposed the celebrated \emph{SYZ conjecture} asserting that mirror manifolds should admit dual Lagrangian torus fibrations, and the interchange of A-branes (Lagrangian cycles) on $Y$ and B-branes (holomorphic cycles) on $\check{Y}$ should be carried out by Fourier-Mukai type transformations (which we call \emph{SYZ transformations}). This suggests a way to construct a canonical and geometric functor
\begin{equation*}
\mathcal{F}:D^b\textrm{Fuk}(Y)\to D^b(\check{Y})
\end{equation*}
implementing the HMS equivalence when a mirror pair $Y$ and $\check{Y}$ are realized as dual Lagrangian torus fibrations.

In the case of dual Lagrangian torus bundles, the construction of such a functor was done by Arinkin-Polishchuk \cite{AP01} and Leung-Yau-Zaslow \cite{LYZ00}. It is a real version of the Fourier-Mukai transform on families of real tori and it transforms a Lagrangian section into a holomorphic line bundle. This construction has then been applied, implicitly or explicitly, to the study of HMS for elliptic curves \cite{AP01} and projective toric manifolds \cite{Ab09,F08,FLTZ08,C09}.

It is desirable to generalize the constructions of \cite{AP01, LYZ00} to more general cases. A serious obstacle in doing so is the presence of singular fibers in Lagrangian torus fibrations. In this case the mirror manifold cannot be obtained by fiberwise dualization ($T$-duality) alone. As shown by Kontsevich-Soibelman \cite{KS04}, Gross-Siebert \cite{GS07} and Auroux \cite{A07,A09}, one needs to incorporate \emph{instanton corrections} and modify the gluing of the complex structure on the mirror space accordingly. From this perspective, in order to construct a geometric functor realizing the HMS equivalence in general, we must take instanton corrections into account. The goal of this paper is to construct an instanton-corrected geometric functor (using SYZ transformations) which realizes the HMS equivalence for $A_n$-resolutions.

Let $Y$ be the affine hypersurface
\begin{equation*}
Y=\{(u,v,z)\in\bC^2\times\bC^\times|uv=f(z)\},
\end{equation*}
where $f(z)\in\bC[z]$ is a degree $n+1$ polynomial with distinct zeros. We equip $Y$ with the symplectic structure
\begin{equation*}
\omega=-\frac{\sqrt{-1}}{2}\Bigg(du\wedge d\bar{u}+dv\wedge d\bar{v}+\frac{dz\wedge d\bar{z}}{|z|^2}\Bigg)\Bigg|_Y.
\end{equation*}
We will show that the mirror of $(Y,\omega)$ is given by the complement $\check{Y}$ of a hypersurface in the crepant resolution $X\to\bC^2/\bZ_{n+1}$ of the $A_n$-singularity. Then our main result is the following (see Theorem \ref{main_thm} and Corollary \ref{main_cor})\\

\noindent\textbf{Main Theorem.} \emph{The SYZ transformations give an equivalence of triangulated categories
\begin{equation*}
\mathcal{F}:D^b\textrm{Fuk}_0(Y,\omega)\overset{\cong}{\longrightarrow}D^b_0(\check{Y}).
\end{equation*}
Here, $D^b\textrm{Fuk}_0(Y,\omega)$ is the derived Fukaya category generated by an $A_n$-configuration of graded Lagrangian spheres in $Y$ and $D^b_0(\check{Y})$ is a full triangulated subcategory of the derived category of coherent sheaves on $\check{Y}$ supported at the exceptional divisor.}\\

To define the geometric functor $\mathcal{F}$ and prove this theorem, we first need to construct the instanton-corrected mirror of $(Y,\omega)$ using SYZ. Following \cite{Go01,G01}, we construct a Lagrangian torus fibration on $Y$. This fibration has singular fibers, but the discriminant locus and the structure of the singular fibers can be described explicitly. According to the general philosophy of the SYZ proposal \cite{SYZ96}, the mirror $\check{Y}$ should be constructed using a $T$-duality (i.e. fiberwise dualization) modified by suitable instanton corrections (counting of holomorphic disks bounded by the Lagrangian torus fibers), as have been done in \cite{A07,A09,C10,CLL11,LLW10}. We demonstrate this construction in details in Section \ref{Mirror}.

The upshot here is that the constructions of \cite{AP01, LYZ00} is compatible with instanton corrections and hence can naturally be generalized in this situation. We are thus able to construct the desired explicit geometric functor $\mathcal{F}$ using SYZ transformations. To show that this realizes the HMS equivalence, we have to compute the SYZ transformations of certain Lagrangian submanifolds. The key observation is that although the connection over the mirror cycle is difficult to compute, the holomorphic structure of the underlying line bundle can be determined rather explicitly. See Section \ref{HMS} for the precise constructions and statements.

Section \ref{SYZ_transform} contains a review and a slight generalization of the constructions of SYZ transformations in \cite{AP01,LYZ00}. In Section \ref{Higher_dim}, we comment on the higher dimensional generalizations of our constructions.

During the final stage of the preparation of this paper, we learned of a recent preprint by Pascaleff \cite{P11} working on a similar problem for the complement of a binodal cubic curve in $\bP^2$. However, we shall point out that the approach and methods of that paper are quite different from ours; in particular, SYZ transformations were not explicitly used there.

\section*{Acknowledgment}
The author would like to thank Kazushi Ueda for useful discussions during the GCOE Conference ``Derived Categories 2011 Tokyo" held in the University of Tokyo in January 2011. He would also like to thank the referee for a careful reading and thoughtful suggestions. Part of this work was done when the author was visiting IH\'ES and working as a project researcher at IPMU, University of Tokyo. He is grateful to both institutes for financial support and providing an excellent research environment. This research was partially supported by a RGC Direct Grant for Research 2011/2012 (Project ID: CUHK2060434).

\section{SYZ transformations}\label{SYZ_transform}

In this section, we review the constructions of SYZ transformations by Arinkin-Polishchuk \cite{AP01} and Leung-Yau-Zaslow \cite{LYZ00}. We will mainly follow \cite{LYZ00}. For a much more detailed description of mirror symmetry without instanton corrections, see Leung \cite{L05}.

Let $B$ be an $n$-dimensional integral affine manifold. This means that the transition maps lie inside the group of affine linear transformations $\textrm{Aff}(\bZ^n):=\textrm{GL}_n(\bZ)\rtimes\bZ^n$. Let $\Lambda\subset TB$ be the family of lattices locally generated by $\partial/\partial x_1,\ldots,\partial/\partial x_n$, where $x_1,\ldots,x_n$ are local affine coordinates on $B$. Dually, let $\Lambda^\vee\subset T^*B$ be the family of lattices locally generated by $dx_1,\ldots,dx_n$. Consider the manifolds
\begin{eqnarray*}
X(B) & := & TB/\Lambda,\\
\check{X}(B) & := & T^*B/\Lambda^\vee.
\end{eqnarray*}

Observe that $X(B)$ is naturally a complex manifold: Let $y_1,\ldots,y_n$ be the fiber coordinates on $TB$, i.e. $(x_1,\ldots,x_n,y_1,\ldots,y_n)\in TB$ denotes the tangent vector $\sum_{j=1}^n y_j\partial/\partial x_j$ at the point $(x_1,\ldots,x_n)\in B$. Then the complex coordinates on $X(B)$ are given by
\begin{equation*}
z_j:=\exp2\pi(x_j+\sqrt{-1}y_j),\ j=1,\ldots,n.
\end{equation*}

On the other hand, $\check{X}(B)$ is naturally a symplectic manifold: Let $\xi_1,\ldots,\xi_n$ be the fiber coordinates on $T^*B$, i.e. $(x_1,\ldots,x_n,\xi_1,\ldots,\xi_n)\in T^*B$ denotes the cotangent vector $\sum_{j=1}^n\xi_jdx_j$ at the point $(x_1,\ldots,x_n)\in B$. Then the canonical symplectic structure on $\check{X}(B)$ is given by
\begin{equation*}
\omega:=\sum_{j=1}^n dx_j\wedge d\xi_j.
\end{equation*}

The projection map
\begin{equation*}
\rho:\check{X}(B)\to B
\end{equation*}
is a Lagrangian torus fibration on $\check{X}(B)$. The torus bundle
\begin{equation*}
\check{\rho}:X(B)\to B
\end{equation*}
is the fiberwise dualization of $\rho$. The SYZ proposal \cite{SYZ96} suggests that we should view $X(B)$ and $\check{X}(B)$ as a mirror pair. This is mirror symmetry without quantum corrections.\footnote{We can also equip $\check{X}(B)$ with a compatible complex structure and $X(B)$ with a compatible symplectic structure making both $\check{X}(B)$ and $X(B)$ K\"ahler manifolds; see \cite{L05,CL10b}.}

In Leung-Yau-Zaslow \cite{LYZ00} (see also Arinkin-Polishchuk \cite{AP01}), a Fourier-Mukai-type transformation $\mathcal{F}$ carrying A-branes (i.e. Lagrangian cycles) on $\check{X}(B)$ to B-branes (i.e. holomorphic cycles) on $X(B)$ is constructed. $\mathcal{F}$ is an example of SYZ transformations \cite{CL10a,CL10b}. Let us briefly review their construction.

The key idea is to view the dual $T^*$ of a torus $T$ as the moduli space of flat $U(1)$-connections on the trivial line bundle $\underline{\bC}:=\bC\times T$ over $T$. This is indeed the basis underlying the SYZ proposal \cite{SYZ96}. Now, if we have a section $L=\{(x,\xi(x))\in\check{X}(B)|x\in B\}$ of the fibration $\rho:\check{X}(B)\to B$. Then each point $(x,\xi(x))\in L$ corresponds to a flat $U(1)$-connection $\nabla_{\xi(x)}$ over the dual fiber $\check{\rho}^{-1}(x)\subset X(B)$. The family of connections $\{\nabla_{\xi(x)}|x\in B\}$ patch together to give a $U(1)$-connection over $X(B)$. By a straightforward calculation, one can show that the connection defines a holomorphic structure on a holomorphic line bundle over $X(B)$ (i.e. the $(0,2)$-part of the curvature two-form is trivial) if and only if the section $L\subset\check{X}(B)$ is Lagrangian. See \cite[Section 3.1]{LYZ00} for a detailed explanation.

For the purpose of this paper, we need a generalization of this construction. Consider a submanifold $S\subset B$ locally defined by $x_j=c_j$ for $j=k+1,\ldots,n$, where $c_j\in\bR$, $j=k+1,\ldots,n$, are constants. Consider a subspace in $\check{X}(B)$ of the form
\begin{eqnarray*}
L & = & \{(x_1,\ldots,x_n,\xi_1,\ldots,\xi_n)\in\check{X}(B)|\\
  &   & \qquad\qquad(x_1,\ldots,x_n)\in S\textrm{ and }\xi_j=\xi_j(x_1,\ldots,x_k)\textrm{ for $j=1,\ldots,k$}\},
\end{eqnarray*}
where $\xi_j=\xi_j(x_1,\ldots,x_k)$, $j=1,\ldots,k$, are $C^\infty$ functions on $S$. Geometrically, $L$ is a translate of the conormal bundle of the integral affine linear subspace $S\subset B$.
\begin{prop}\label{prop2.1}
$L$ is Lagrangian in $\check{X}(B)$ with respect to $\omega$ if and only if
\begin{equation*}
\frac{\partial \xi_j}{\partial x_l}=\frac{\partial \xi_l}{\partial x_j}
\end{equation*}
for any $j,l=1,\ldots,k$.
\end{prop}
\begin{proof}
The restriction of $\omega$ to $L$ is given by
\begin{equation*}
\omega|_L=\sum_{j>l}\Bigg(\frac{\partial \xi_j}{\partial x_l}-\frac{\partial \xi_l}{\partial x_j}\Bigg)dx_j\wedge dx_l.
\end{equation*}
\end{proof}
Restricting the projection map $\rho$ to $L$ gives a $T^{n-k}$-fibration $\rho_L:L\to S$. Denote by $L_x:=\rho_L^{-1}(x)$ the fiber of $\rho_L$ over a point $(x_1,\ldots,x_k,0,\ldots,0)\in S$.

We would like to equip $L$ with a flat $U(1)$-connection. Consider
\begin{equation*}
\nabla:=d+2\pi\sqrt{-1}\left(\sum_{j=1}^ka_jdx_j+\sum_{l=k+1}^nb_ld\xi_l\right),
\end{equation*}
where $a_j=a_j(x_1,\ldots,x_k)$, $j=1,\ldots,k$, are $C^\infty$ functions on $S$ and $b_l\in\bR$, $l=k+1,\ldots,n$, are constants.
\begin{prop}\label{prop2.2}
The connection $\nabla$ over $L$ is flat if and only if
\begin{equation*}
\frac{\partial a_j}{\partial x_l}=\frac{\partial a_l}{\partial x_j} \textrm{ for $j,l=1,\ldots,k$.}
\end{equation*}
\end{prop}
\begin{proof}
The curvature of $\nabla$ is given by
\begin{equation*}
\nabla^2=2\pi\sqrt{-1}\sum_{j>l}\Bigg(\frac{\partial a_j}{\partial x_l}-\frac{\partial a_l}{\partial x_j}\Bigg)dx_l\wedge dx_j.
\end{equation*}
\end{proof}

The SYZ transformation of the A-brane $(L,\nabla)$ on $\check{X}(B)$ should be given by a B-brane $(C,\check{\nabla})$ on $X(B)$ where $C\subset X(B)$ is a complex submanifold and $\check{\nabla}$ is a $U(1)$-connection over $C$.

We start with the construction of $C$. Recall that a point $(x,y)\in X(B)$ defines a flat $U(1)$-connection on the trivial line bundle $\underline{\bC}$ over the torus fiber $\rho^{-1}(x)$. More precisely, $y=(y_1,\ldots,y_n)\in\check{\rho}^{-1}(x)$ corresponds to the connection
\begin{equation*}
\nabla_y:=d+2\pi\sqrt{-1}\sum_{j=1}^ny_jd\xi_j
\end{equation*}
over $\rho^{-1}(x)$. We define $C\subset X(B)$ to be the set of points $(x,y)\in X(B)$ such that $x\in S$ and the connection $\nabla$ twisted by $\nabla_y$ is trivial when it is restricted to $L_x$. In terms of coordinates, this means
\begin{equation*}
\Bigg(\nabla+2\pi\sqrt{-1}\sum_{j=1}^ny_jd\xi_j\Bigg)\Bigg|_{L_x}=d+2\pi\sqrt{-1}\Bigg(\sum_{j=k+1}^n(b_j+y_j)d\xi_j\Bigg)
\end{equation*}
is trivial, i.e. $y_j=-b_j$ for $j=k+1,\ldots,n$. Hence $C$ is the complex submanifold $C\subset X(B)$ defined by the equations
\begin{equation*}
z_j=A_j\textrm{ for $j=k+1,\ldots,n$},
\end{equation*}
where $A_j=\exp2\pi(c_j-\sqrt{-1}b_j)\in\bC^\times$.

Restricting the projection map $\check{\rho}$ to $C$ gives a $T^k$-fibration $\check{\rho}_C:C\to S$, which is exactly the fiberwise dualization of $\rho_L:L\to S$. Indeed, for $x\in S$, we have an exact sequence
\begin{equation*}
0\longrightarrow T_xS\longrightarrow T_xB\longrightarrow N_xS\longrightarrow0,
\end{equation*}
where $N_xS$ denotes the fiber at $x$ of the normal bundle to $S$. Taking the dual gives
\begin{equation*}
0\longrightarrow N^*_xS\longrightarrow T^*_xB\longrightarrow T^*_xS\longrightarrow0.
\end{equation*}
The fiber of the map $\rho_L$ over $x\in S$ is given by the torus
\begin{equation*}
L_x=\rho_L^{-1}(x)=N^*_xS/(N^*_xS\cap\Lambda^\vee_x),
\end{equation*}
whose dual is given by
\begin{equation*}
T_xS/(T_xS\cap\Lambda_x),
\end{equation*}
which is precisely the fiber $C_x:=\check{\rho}_C^{-1}(x)$ of $\check{\rho}_C$ over $x\in S$.

It remains to construct the connection $\check{\nabla}$. This is done by reversing the construction of $C$. Namely, we define $\check{\nabla}$ to be a connection of the form
\begin{equation*}
\check{\nabla}=d+2\pi\sqrt{-1}\Bigg(\sum_{j=1}^k a_jdx_j+\sum_{j=1}^k\beta_jdy_j\Bigg),
\end{equation*}
where $a_j=a_j(x_1,\ldots,x_k)$, $j=1,\ldots,k$, are the functions appearing in $\nabla$ while $\beta_j=\beta_j(x_1,\ldots,x_k)$, $j=1,\ldots,k$ are some other $C^\infty$ functions on $S$, such that $\check{\nabla}$ twisted by $\nabla_\xi$ is trivial when it is restricted to $C_x$ for any $(x,\xi)\in L$. Here, $\nabla_\xi$ is the flat $U(1)$-connection
\begin{equation*}
\nabla_\xi:=d-2\pi\sqrt{-1}\sum_{j=1}^n\xi_jdy_j
\end{equation*}
over $\check{\rho}^{-1}(x)$ corresponding to a point $(x,\xi)\in\check{X}(B)$. So we are requiring the connection
\begin{equation*}
\Bigg(\check{\nabla}-2\pi\sqrt{-1}\sum_{j=1}^n\xi_jdy_j\Bigg)\Bigg|_{C_x}=d+2\pi\sqrt{-1}\Bigg(\sum_{j=1}^k(\beta_j-\xi_j)dy)_j\Bigg)
\end{equation*}
to be trivial for any $x\in S$. In other words, we must have $\beta_j=\xi_j$ for $j=1,\ldots,k$. Hence the connection $\check{\nabla}$ is given by
\begin{equation*}
\check{\nabla}=d+2\pi\sqrt{-1}\Bigg(\sum_{j=1}^k a_jdx_j+\sum_{j=1}^k\xi_jdy_j\Bigg),
\end{equation*}
where $\xi_j=\xi_j(x_1,\ldots,x_k)$, $j=1,\ldots,k$, are the functions defining $L\subset\check{X}(B)$. Notice that we do not transform the base directions.

The $(0,2)$-part $F^{(0,2)}$ of the curvature two-form $F$ of $\check{\nabla}$ is given by
\begin{eqnarray*}
F^{(0,2)} & = & -\frac{\sqrt{-\pi}}{2}\left[\sum_{j>l}\Bigg(\frac{\partial a_j}{\partial x_l}-\frac{\partial a_l}{\partial x_j}\Bigg)\frac{d\bar{z}_j}{\bar{z}_j}\wedge\frac{d\bar{z}_l}{\bar{z}_l}\right]\\
&  & \quad\qquad+\frac{\pi}{2}\left[\sum_{j>l}
\Bigg(\frac{\partial\xi_j}{\partial x_l}-\frac{\partial\xi_l}{\partial x_j}\Bigg)
\frac{d\bar{z}_j}{\bar{z}_j}\wedge\frac{d\bar{z}_l}{\bar{z}_l}\right].
\end{eqnarray*}
Since $L$ is a Lagrangian submanifold and the connection $\nabla$ is flat, we have the following
\begin{prop}
The $(0,2)$-part $F^{(0,2)}$ of the curvature two-form $F$ of $\check{\nabla}$ is trivial, so the $U(1)$-connection $\check{\nabla}$ defines a holomorphic line bundle $\mathscr{L}$ over $C$.
\end{prop}
\begin{proof}
This follows from Propositions \ref{prop2.1} and \ref{prop2.2}.
\end{proof}
It is easy to see that the triviality of $F^{(0,2)}$ is in fact equivalent to requiring $L$ to be Lagrangian and $\nabla$ to be flat. And in this case, the $(2,0)$-part $F^{(2,0)}$ of the curvature two-form is also trivial, so $F$ is the Chern connection of a certain hermitian metric on $\mathscr{L}$.
\begin{defn}\label{def_SYZ_transform}
We define the SYZ transformation of the A-brane $(L,\nabla)$ on $\check{X}(B)$ to be the B-brane $(C,\check{\nabla})$ on $X(B)$, i.e.
\begin{equation*}
\mathcal{F}(L,\nabla):=(C,\check{\nabla}).
\end{equation*}
We also define the SYZ transformation of the isomorphism class ($L$ up to Hamiltonian isotopy and $\nabla$ up to gauge equivalence) of $(L,\nabla)$ to be $(C,\mathscr{L})$.
\end{defn}

\section{Constructing mirrors by SYZ}\label{Mirror}

Consider the affine hypersurface
\begin{equation*}
Y=\{(u,v,z)\in\bC^2\times\bC^\times|uv=f(z)\},
\end{equation*}
where $f(z)\in\bC[z]$ is a polynomial of degree $n+1$ with distinct zeros. We equip $Y$ with the K\"ahler structure
\begin{equation*}
\omega=-\frac{\sqrt{-1}}{2}\Bigg(du\wedge d\bar{u}+dv\wedge d\bar{v}+\frac{dz\wedge d\bar{z}}{|z|^2}\Bigg)\Bigg|_Y.
\end{equation*}

In this section, we will construct the instanton-corrected mirror of $(Y,\omega)$ by carrying out the SYZ proposal \cite{SYZ96}. This procedure is well-known among experts, and was first carried out by Auroux \cite{A07,A09} and later generalized by Chan-Lau-Leung \cite{CLL11} (see also Chan \cite{C10} and Lau-Leung-Wu \cite{LLW10}).

We start by constructing a Lagrangian torus fibration on $Y$, following Goldstein \cite{Go01} and Gross \cite{G01}. Consider the following Hamiltonian $S^1$-action on $Y$.
\begin{equation*}
e^{2\pi\sqrt{-1}t}\cdot(u,v,z)=(e^{2\pi\sqrt{-1}t}u,e^{-2\pi\sqrt{-1}t}v,z).
\end{equation*}
The associated moment map $\mu:Y\to\bR$ is given by
\begin{equation*}
\mu(u,v,z)=\frac{1}{2}(|u|^2-|v|^2).
\end{equation*}

For $\lambda\in\mu(Y)=\bR$, let
\begin{equation*}
Y_\lambda:=\mu^{-1}(\lambda)/S^1
\end{equation*}
be the reduced space. $Y_\lambda$ is smooth and diffeomorphic to $\bC^\times$ for any $\lambda\in\bR$. The reduced symplectic structure on $Y_\lambda$ is given by (cf. \cite[Section 4]{G01}):
\begin{eqnarray*}
\omega_\lambda & = & -\frac{\sqrt{-1}}{2}\Bigg(\frac{df\wedge d\bar{f}}{2\sqrt{\lambda^2+|f|^2}}+\frac{dz\wedge d\bar{z}}{|z|^2}\Bigg)\\
& = & -\frac{\sqrt{-1}}{2}\Bigg(\frac{|f'|^2}{2\sqrt{\lambda^2+|f|^2}}+\frac{1}{|z|^2}\Bigg)dz\wedge d\bar{z}.
\end{eqnarray*}

The reduced symplectic manifold $(Y_\lambda,\omega_\lambda)$ is symplectomorphic to $(\bC^\times,\omega_0)$, where $\omega_0=-\frac{\sqrt{-1}}{2}\frac{dz\wedge d\bar{z}}{|z|^2}$ is the standard K\"ahler structure on $\bC^\times$. Pulling back the log map
\begin{equation*}
\textrm{Log}:\bC^\times\to\bR,\ z\mapsto\log|z|
\end{equation*}
by the symplectomorphism gives a Lagrangian $S^1$-fibration on $Y_\lambda$. Then by taking the preimages of the fibers in $\mu^{-1}(\lambda)$, we obtain a Lagrangian torus fibration $\rho:Y\to\bR^2$ on $Y$ defined by
\begin{equation*}
\rho(u,v,z)=\Bigg(\log|z|,\frac{1}{2}(|u|^2-|v|^2)\Bigg).
\end{equation*}
The fact that this is a Lagrangian fibration can also be checked by direct computations. Denote by $T_{s,\lambda}$ the Lagrangian torus fiber of $\rho$ in $Y$ over the point $(s,\lambda)\in\bR^2$.

The map $\rho$ is not regular precisely at the fixed points of the $S^1$-action on $Y$: $\{(u,v,z)\in\bC^2\times\bC^\times|u=0,v=0,f(z)=0\}$. Let $\Delta:=\{a_0,a_1,\ldots,a_n\}\subset\bC^\times$ be the set of zeros of $f(z)$. Then a fiber $T_{s,\lambda}$ is singular (with nodal singularities) if and only if $\lambda=0$ and $s=s_i:=\log|a_i|$ for some $i=0,1,\ldots,n$. We assume that $s_0<s_1<\ldots<s_n$ so that each singular fiber of $\rho$ contains only one nodal singular point.

The instanton corrections are contributions from holomorphic disks bounded by the Lagrangian torus fibers of $\rho$. By the maximum principle, the image of any nonconstant holomorphic map $\varphi:D^2\to Y$ from the unit disk $D^2$ to $Y$ with boundary on a torus fiber of $\rho$ must be contained in a fiber of the projection map
\begin{equation*}
\pi:Y\to\bC^\times,\ (u,v,z)\mapsto z,
\end{equation*}
and the fiber must be singular because a smooth fiber of $\pi$, which is a smooth conic in $\bC^2$, does not contain any nonconstant holomorphic disk. Therefore, a Lagrangian torus fiber $T_{s,\lambda}$, which is a family of circles over the circle $\{z\in\bC^\times||z|=e^s\}\subset\bC^\times$, bounds a holomorphic disk if and only if $s=s_i$ for some $i=0,1,\ldots,n$.

We summarize our discussion by the following proposition.
\begin{prop}\label{SYZ_fibration}
Let $B:=\bR^2$. The map $\rho:Y\to B$ defined by
\begin{equation*}
\rho(u,v,z)=\Bigg(\log|z|,\frac{1}{2}(|u|^2-|v|^2)\Bigg)
\end{equation*}
is a Lagrangian torus fibration on $Y$. The discriminant locus of the fibration is given by the finite set
\begin{equation*}
\Gamma:=\{(s,\lambda)\in B|\lambda=0,\ s=s_i\textrm{ for some $i=0,1,\ldots,n$}\}.
\end{equation*}
Let $B^{sm}:=B\setminus\Gamma$. The fiber $T_{s,\lambda}$ over a point $(s,\lambda)\in\Gamma$ is singular with one nodal singularity, and the fiber $T_{s,\lambda}$ over a point $(s,\lambda)\in B^{sm}$ is a smooth Lagrangian torus in $Y$. The locus of Lagrangian torus fibers which bound nonconstant holomorphic disks is given by the union of vertical lines:
\begin{equation*}
H:=\{(s,\lambda)\in B|s=s_i\textrm{ for some $i=0,1,\ldots,n$}\}.
\end{equation*}
We call $H$ the \emph{wall} in $B$.
\end{prop}

The SYZ proposal \cite{SYZ96} suggests that the mirror should be given by the moduli space of Lagrangian torus fibers equipped with flat $U(1)$-connections. More precisely, let $\check{Y}_0$ be the moduli space of pairs $(T_{s,\lambda},\nabla)$ where $(s,\lambda)\in B^{sm}$ and $\nabla$ is a flat $U(1)$-connection (up to gauge equivalence) on the trivial line bundle $\underline{\bC}:=\bC\times T_{s,\lambda}$ over the smooth Lagrangian torus $T_{s,\lambda}$. We call $\check{Y}_0$ the semi-flat mirror of $(Y,\omega)$.

The Lagrangian torus fibration $\rho:Y\to B$ induces an integral affine structure on $B^{sm}$. This is usually called the \textit{symplectic affine structure}. Recall that the affine structure being integral means that the transition maps on $B^{sm}$ lie inside the integral affine linear group $\textrm{Aff}(\bZ^2):=\textrm{GL}_2(\bZ)\rtimes\bZ^2$. As in Section \ref{SYZ_transform}, denote by $\Lambda\subset TB^{sm}$ the family of lattices locally generated by $\partial/\partial x_1,\partial/\partial x_2$, where $x_1,x_2$ are local affine coordinates on $B^{sm}$. $\Lambda$ is well-defined because of integrality of the affine structure. Then $\check{Y}_0$ is topologically the quotient $TB^{sm}/\Lambda$ of $TB^{sm}$ by $\Lambda$, and hence admits a naturally defined complex structure, where the local complex coordinates are given by exponentiations of complexifications of local affine coordinates on $B^{sm}$.

When there are no singular fibers, the complex structure on $\check{Y}_0$ is globally defined and it gives the genuine mirror manifold. However, in our example, the complex structure is not globally defined on $\check{Y}_0$ due to nontrivial monodromy of the affine structure around each point $(s_i,0)$ in the discriminant locus $\Gamma$. This can be described more explicitly as follows. For $i=0,1,\ldots,n$, consider the strip $B_i:=(s_{i-1},s_{i+1})\times\bR$ (where we set $s_{-1}:=-\infty$ and $s_{n+1}:=+\infty$). A covering of $B^{sm}$ is given by the open sets
\begin{equation*}
U_i:=B_i\setminus[s_i,s_{i+1})\times\{0\},\ V_i:=B_i\setminus(s_{i-1},s_{i}]\times\{0\}.
\end{equation*}
$i=0,1,\ldots,n$. The intersection $U_i\cap V_i$ consists of two components: $U_i\cap V_i=B_i^+\cup B_i^-$ where $B_i^+$ (resp. $B_i^-$) corresponds to the component in which $\lambda>0$ (resp. $\lambda<0$).

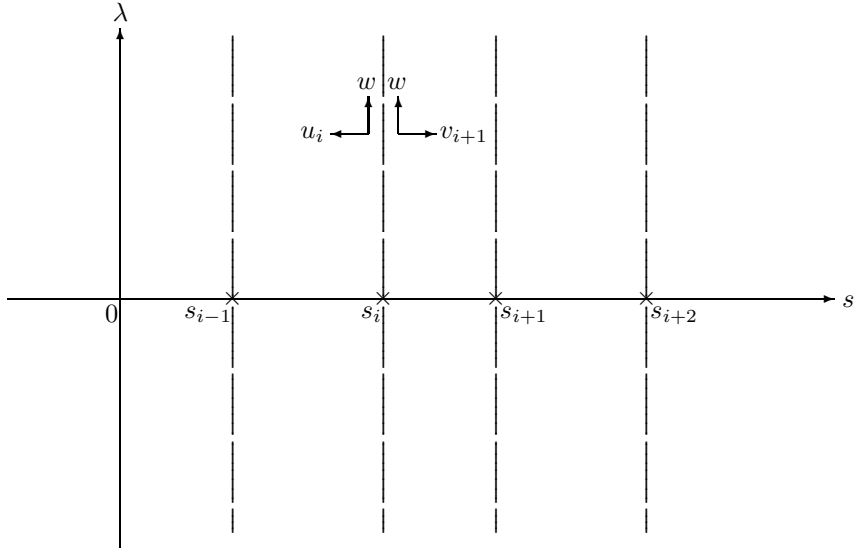
\begin{figure}[ht]
\setlength{\unitlength}{1mm}
\begin{picture}(100,75)
\put(10,2){\vector(0,1){69}} \put(9,72){$\lambda$} \put(-5,35){\vector(1,0){110}} \put(106,34){$s$} \put(8,32){$0$}
\curve(25,70, 25,62) \curve(25,61, 25,53) \curve(25,52, 25,44) \curve(25,43, 25,35) \curve(25,34, 25,26) \curve(25,25, 25,17) \curve(25,16, 25,8) \curve(25,7, 25,4) \put(23.6,34.2){$\times$} \put(18.5,32.5){$s_{i-1}$}
\curve(45,70, 45,62) \curve(45,61, 45,53) \curve(45,52, 45,44) \curve(45,43, 45,35) \curve(45,34, 45,26) \curve(45,25, 45,17) \curve(45,16, 45,8) \curve(45,7, 45,4) \put(43.6,34.2){$\times$} \put(42,32.5){$s_i$}
\curve(60,70, 60,62) \curve(60,61, 60,53) \curve(60,52, 60,44) \curve(60,43, 60,35) \curve(60,34, 60,26) \curve(60,25, 60,17) \curve(60,16, 60,8) \curve(60,7, 60,4) \put(58.6,34.2){$\times$} \put(60.5,32.5){$s_{i+1}$}
\curve(80,70, 80,62) \curve(80,61, 80,53) \curve(80,52, 80,44) \curve(80,43, 80,35) \curve(80,34, 80,26) \curve(80,25, 80,17) \curve(80,16, 80,8) \curve(80,7, 80,4) \put(78.6,34.2){$\times$} \put(80.5,32.5){$s_{i+2}$}
\put(43,57){\vector(0,1){5}} \put(41.5,63){$w$} \put(43,57){\vector(-1,0){5}} \put(34,56.5){$u_i$}
\put(47,57){\vector(0,1){5}} \put(45.5,63){$w$} \put(47,57){\vector(1,0){5}} \put(52.5,56.5){$v_{i+1}$}
\end{picture}
\caption{The base affine manifold $B$.}\label{base_B}
\end{figure}

On $U_i$ (resp. $V_i$), denote by $u_i$ (resp. $v_{i+1}$) the exponentiation of the complexification\footnote{Our convention is: For a real number $x\in\bR$, its complexification is given by $-x+\sqrt{-1}y$.} of the left-pointing (resp. right-pointing) affine coordinate. Also, denote by $w$ the exponentiation of the complexification of the upward-pointing affine coordinate. See Figure \ref{base_B}.

The coordinate $w$ is globally defined. Indeed if we regard $\check{Y}_0$ as the moduli space of pairs $(T_{s,\lambda},\nabla)$ and denote by $\alpha\in\pi_2(Y,T_{s_i,\lambda})$ the class of a holomorphic disk in $Y$ bounded by $T_{s_i,\lambda}$ for some $i=0,1,\ldots,n$, then the function $w$ can be written as
\begin{equation*}
w(T_{s,\lambda},\nabla)=\left\{\begin{array}{ll}
                        \exp(-\int_\alpha\omega)\textrm{hol}_\nabla(\partial\alpha) & \textrm{ for $\lambda>0$},\\
                        \exp(-\int_{-\alpha}\omega)\textrm{hol}_\nabla(\partial(-\alpha)) & \textrm{ for $\lambda<0$},
                        \end{array}\right.
\end{equation*}
where $\textrm{hol}_\nabla(\partial\alpha)$ is the holonomy of the flat $U(1)$-connection $\nabla$ around the boundary $\partial\alpha\in\pi_1(T_{s_i,\lambda})$, so that $|w|=e^{-\lambda}$ and $\int_\alpha\omega=|\lambda|$. Note that the class $\alpha$ changes to $-\alpha$ when one moves from positive $\lambda$ to negative $\lambda$ and $\partial(-\alpha)=-\partial\alpha\in\pi_1(T_{s_i,\lambda})$ (cf. Auroux \cite{A07,A09}).

Since the monodromy of the affine structure going counter-clockwise around $(s_i,0)\in\Gamma$ is given by the matrix
\begin{equation*}
\left(\begin{array}{ll}
        1 & 1\\
        0 & 1
      \end{array}\right),
\end{equation*}
we glue the coordinates $(u_i,w)$ and $(v_{i+1},w)$ on $U_i\cap V_i$ accordingly by
\begin{eqnarray*}
u_i=v_{i+1}^{-1} & \textrm{ on $B_i^+$,}\\
u_i=v_{i+1}^{-1}w & \textrm{ on $B_i^-$.}
\end{eqnarray*}
Hence, the monodromy of the complex coordinates going counter-clockwise around the point $(s_i,0)$ is given by
\begin{equation*}
u_i\mapsto u_iw,\ v_{i+1}^{-1}\mapsto v_{i+1}^{-1}w.
\end{equation*}
In particular, the complex coordinates on the components $TB_i/\Lambda$ do not glue together to form a globally defined complex structure on $\check{Y}_0$. Notice that the gluing of the coordinates $(v_{i+1},w)$ and $(u_{i+1},w)$ is given by $v_{i+1}^{-1}=u_{i+1}$ in the intersection $V_i\cap U_{i+1}=(s_i,s_{i+1})\times\bR$, and this does not give rise to nontrivial monodromy.

Now we need to modify the gluing of the complex charts on $\check{Y}_0$ by incorporating instanton corrections. Following the construction in \cite{A07,A09,C10,CLL11,LLW10} (which in turn are special cases of the general construction in \cite{KS04,GS07}), the gluing should be modified as:
\begin{eqnarray}\label{wc+}
u_i=v_{i+1}^{-1}+v_{i+1}^{-1}w =& v_{i+1}^{-1}(1+w) & \textrm{ on $B_i^+$,}\\
u_i=v_{i+1}^{-1}w+v_{i+1}^{-1} =& v_{i+1}^{-1}w(1+w^{-1}) & \textrm{ on $B_i^-$.}\label{wc-}
\end{eqnarray}

Geometrically, the term $v_{i+1}^{-1}w$ in the first formula (\ref{wc+}) should be regarded as multiplying $v_{i+1}^{-1}$ by $w$ where $w$ corresponds to the unique nonconstant holomorphic disk bounded by the Lagrangian torus $T_{0,\lambda}$ whose area is given by
$$\lambda=-\log|w|>0.$$
This means that we are correcting the term $v_{i+1}^{-1}$ by adding the contribution from the holomorphic disk bounded by $T_{0,\lambda}$ when we cross the upper half of the wall $\{s_i\}\times\bR\subset H$. Similarly, the term $v_{i+1}^{-1}$ in the second formula (\ref{wc-}) should be read as multiplying $v_{i+1}^{-1}w$ by $w^{-1}$ where $w^{-1}$ corresponds to the unique nonconstant holomorphic disk bounded by the Lagrangian torus $T_{0,\lambda}$ whose area is given by
$$-\lambda=\log|w|=-\log|w^{-1}|>0.$$
So we are correcting $v_{i+1}^{-1}w$ again by the disk bounded by $T_{0,\lambda}$ when we cross the lower half of the wall $\{s_i\}\times\bR\subset H$.

Alternatively, one may partially compactify $Y$ by allowing $z\in\bC^\times$ to be zero and replacing $\rho$ by the map
$$(u,v,z)\mapsto(|z|,(|u|^2-|v|^2)/2).$$
Then the base becomes an affine manifold with both singularities and boundary; more precisely, it is given by a right half-space $\bar{B}$ in $\bR^2$ and the boundary $\partial\bar{B}$ corresponds to the divisor in $Y$ defined by $z=0$. In this setting, the formulas (\ref{wc+}), (\ref{wc-}) can be viewed as wall-crossing formulas for the counting of Maslov index two holomorphic disks in $Y$ bounded by Lagrangian torus fibers and each of the terms $u_i$, $v_{i+1}^{-1}$ and $v_{i+1}^{-1}w$ can be expressed as
\begin{equation*}
\exp\Bigg(-\int_\beta\omega\Bigg)\textrm{hol}_\nabla(\partial\beta),
\end{equation*}
for a suitable relative homotopy class $\beta\in\pi_2(Y,T_{s,\lambda})$ of Maslov index two (cf. Auroux \cite{A07,A09}).

In any case, the key consequence of this modification of gluing is that it cancels the monodromy and defines global complex coordinates $u_i,v_i,w$ on $TB_i/\Lambda$ (topologically, $TB_i/\Lambda\cong(\bC^\times)^2$) related by
\begin{equation}\label{gluing}
u_i=v_i^{-1},\ u_{i+1}=v_{i+1}^{-1}\textrm{ and }u_iv_{i+1}=1+w.
\end{equation}
The instanton-corrected mirror $\check{Y}$ is obtained by gluing together $TB_i/\Lambda$, $i=0,1,\ldots,n$, according to (\ref{gluing}) and analytic continuation.\footnote{By our definition, the mirror manifold $\check{Y}$ will have ``gaps" because for instance $u,v,w$ are $\bC^\times$-valued. A natural way to ``complete" the mirror manifold and fill out its gaps is by rescaling the symplectic structure of $Y$ and analytic continuation. See \cite[Section 4.2]{A07} for a discussion of this ``renormalization" process.} On the intersection $(TB_{i-1}/\Lambda)\cap(TB_i/\Lambda)$, we have
\begin{equation*}
u_i=v_i^{-1},\ u_iv_{i+1}=1+w=u_{i-1}v_i.
\end{equation*}

We observe that the above equations give precisely the gluing of complex charts in the toric resolution of the $A_n$-singularity. Let us recall the construction of the $A_n$-resolution as a toric surface. Let $N=\bZ^2$ and $M:=\textrm{Hom}(N,\bZ)\cong\bZ^2$. Let $N_R:=N\otimes_\bZ R$ and $M_R:=M\otimes_\bZ R$ for any $\bZ$-module $R$ and denote by $\langle\cdot,\cdot\rangle:M_R\times N_R\to R$ the natural pairing. Consider the 2-dimensional fan $\Sigma\in N_\bR=\bR^2$ generated by $\{a_i:=(i,1)\in N|i=0,1,\ldots,n+1\}$. The toric surface $X_\Sigma$ defined by this fan is Calabi-Yau (i.e. the canonical line bundle $K_{X_\Sigma}$ is trivial) and it gives a crepant resolution
\begin{equation*}
g:X_\Sigma\to\bC^2/\bZ_{n+1}
\end{equation*}
of the $A_n$-singularity $\bC^2/\bZ_{n+1}$. For $i=0,1,\ldots,n$, let $\sigma_i\in\Sigma$ be the cone generated by $a_i$ and $a_{i+1}$. To each cone $\sigma_i$ we associate an affine toric variety $X_{\sigma_i}$.

A lattice point $b\in M$ corresponds to a function $\chi^b$ on $N_{\bC^\times}$; in particular, the generators $b_{i+1}:=(-1,i+1),b_i:=(1,-i)\in M$ of the dual cone $\sigma_i^\vee$ gives coordinates on the affine toric variety $X_{\sigma_i}$:
\begin{equation*}
u_i:=\chi^{b_i},\ v_{i+1}:=\chi^{b_{i+1}}.
\end{equation*}
The toric variety $X_\Sigma$ is obtained by gluing the affine toric varieties $X_{\sigma_i}$ according to:
\begin{equation*}
u_i=v_i^{-1}, u_iv_{i+1}=\chi^{(0,1)}=u_{i+1}v_i.
\end{equation*}
This is precisely the gluing of complex charts in the mirror manifold $\check{Y}$. However, in $\check{Y}$, the function $u_iv_{i+1}=1+w$ is never equal to 1 since $w$ is $\bC^\times$-valued. This corresponds to removing the hypersurface $D=\{h:=\chi^{(0,1)}=1\}$ in $X_\Sigma$.
\begin{thm}
By carrying out the procedure outlined in \cite[Section 2.3]{CLL11}, we obtain the complex manifold
\begin{equation*}
\check{Y}:=X_\Sigma\setminus D,
\end{equation*}
where $D$ is the hypersurface in $X_\Sigma$ defined by $h=1$. We call $\check{Y}$ the \emph{instanton-corrected mirror} of $(Y,\omega)$.
\end{thm}
For $i=1,\ldots,n$, denote by $S_i$ the interval $(s_{i-1},s_i)\times\{0\}\subset B$. Then (the closure of) the quotient $TS_i/TS_i\cap\Lambda$ in $\check{Y}$ is one of the $n$ exceptional divisors (each is a $(-2)$-curve) in $X_\Sigma$. We denote this exceptional divisor by $E_i\subset\check{Y}$.

\section{HMS for $A_n$-resolutions}\label{HMS}

In this section, we introduce a class of Lagrangian submanifolds which are Lefschetz thimbles in $(Y,\omega)$ following Khovanov-Seidel \cite{KS02}.\footnote{The construction of these Lagrangian submanifolds was originally due to Donaldson \cite{D00}, and has since been extensively used in e.g. \cite{KS02,S01,S08,P11}.} We then compute the SYZ transformations of these Lagrangian submanifolds and show that this defines a triangulated equivalence between a derived Fukaya category of $(Y,\omega)$ and a full triangulated subcategory of the derived category of coherent sheaves on the mirror $\check{Y}$.

Consider the projection map
\begin{equation*}
\pi:Y\to\bC^\times,\ (u,v,z)\mapsto z.
\end{equation*}
Each fiber is a conic in $\bC^2$ which becomes singular when $z$ is a zero of $f(z)$. Denote by $\Delta=\{a_0,a_1,\ldots,a_n\}\subset\bC^\times$ the set of zeros of $f(z)$. For any path $\gamma:[0,1]\to\bC^\times$ which is either an embedding or a loop (i.e. $\gamma(0)=\gamma(1)$), let
\begin{equation*}
L_\gamma:=\{(u,v,\gamma(t)):|u|=|v|,\ t\in[0,1]\}.
\end{equation*}
\begin{lem}
$L_\gamma$ is Lagrangian.
\end{lem}
This can be checked by direct computations or by observing that $L_\gamma$ is invariant under parallel transport with respect to the horizontal distribution (given by symplectic orthogonal to the fiber) induced by the symplectic fibration $\pi:Y\to\bC^\times$.

When $\gamma_s$ is the circle $\{\exp(s+2\pi\sqrt{-1}t)\in\bC^\times|t\in[0,1]\}$ in $\bC^\times$ with radius $e^s$ centered at the origin, $L_{\gamma_s}$ is nothing but the fiber $T_{s,0}$ of the fibration $\rho:Y\to B$ over the point $(s,0)\in B$. Suppose that $s_{i-1}<s<s_i$ for some $i=1,\ldots,n$. Then $L_{\gamma_s}=T_{s,0}$ is a smooth Lagrangian torus in $Y$. Let $\nabla$ be a flat $U(1)$-connection on the trivial line bundle $\underline{\bC}$ over $L_{\gamma_s}$. Then the SYZ transformation of $(L_{\gamma_s},\nabla)$ is given by the skyscraper sheaf $\mathcal{O}_p$ over a point $p$ in the exceptional divisor $E_i\subset\check{Y}$.

On the other hand, let $\gamma:[0,1]\to\bC^\times$ be an embedded path such that $\gamma(0),\gamma(1)\in\Delta$ and $\gamma(t)\not\in\Delta$ for $0<t<1$. We call such a path \textit{admissible} (\cite{KS02}; see also \cite{IUU10}). In this case, $L_\gamma$ is the Lefschetz thimble, i.e. the family of vanishing cycles over $\gamma$. Since $L_\gamma$ is diffeomorphic to a sphere $S^2$, it is an exact Lagrangian submanifold (i.e. when we write $\omega=d\alpha$ for some 1-form $\alpha$, the restriction $\alpha|_{L_\gamma}$ is exact). For two isotopic admissible paths $\gamma$ and $\gamma'$, the Lagrangians $L_\gamma$ and $L_{\gamma'}$ are Hamiltonian isotopic \cite[Section 5.3]{IUU10}.
\begin{defn}
Given $i\in\{1,\ldots,n\}$ and an admissible path $\gamma:[0,1]\to\bC^\times$ going from $a_{i-1}$ to $a_i$. We call $\gamma$ strongly admissible if
\begin{equation*}
|\textrm{Im}(\gamma)\cap\{z\in\bC^\times||z|=e^s\}|=1
\end{equation*}
for any $s_{i-1}\leq s\leq s_i$.
\end{defn}
In other words, an admissible path $\gamma$ from $a_{i-1}$ to $a_i$ is strongly admissible if it intersects each circle in the annulus $\{z\in\bC^\times|e^{s_{i-1}}\leq|z|\leq e^{s_i}\}$ centered at the origin at one point.

Given such a path $\gamma$, the Lagrangian submanifold $L_\gamma$ intersects the Lagrangian torus fiber $T_{s,0}$ over each point in the interval $S_i=(s_{i-1},s_i)\times\{0\}\subset B$ in a circle. So $L_\gamma^\circ:=L_\gamma\setminus\{(0,0,a_{i-1}),(0,0,a_i)\}$ is of the type we consider in Section \ref{SYZ_transform}, i.e. $L_\gamma^\circ$ is a translate of the conormal bundle of $S_i\subset B$. We equip $L_\gamma$ with the flat $U(1)$-connection $\nabla_0=d+\frac{\sqrt{-1}}{2}du$.\footnote{Since $L_\gamma$ is a sphere, $\nabla_0$ is in fact gauge-equivalent to the trivial connection $d$.}

According to Definition \ref{def_SYZ_transform}, the SYZ transformation of $(L_\gamma,\nabla_0)$ is given by a B-brane $(C,\check{\nabla})$ where $C$ is a complex submanifold in $\check{Y}$ and $\check{\nabla}$ is a $U(1)$-connection on a holomorphic line bundle $\mathscr{L}$ over $C$. The main goal of this section is to compute $C$ and $\mathscr{L}$.

Consider the strip $W_i:=(s_{i-1},s_i)\times\bR$ in $B$. Using the notations in Section \ref{Mirror}, we have $W_i=V_{i-1}\cap U_i$. Let
\begin{equation*}
\Theta_i:T^*W_i/T^*W_i\cap\Lambda^\vee\to\rho^{-1}(W_i)
\end{equation*}
be a fiber-preserving symplectomorphism, so that
\begin{equation*}
\Theta_i^*(\omega)=dx_1\wedge d\xi_1+dx_2\wedge d\xi_2,
\end{equation*}
where $(x_1,x_2)\in W_i$ are the (symplectic) affine coordinates on $W_i$ and $(\xi_1,\xi_2)$ are fiber coordinates on $T^*W_i$. One of the affine coordinates, say $x_2$, on $W_i$ is given by the original coordinate $-\lambda$ (where $\lambda$ is the moment map of the $S^1$-action on $Y$). In terms of the dual coordinates $(x_1,x_2,y_1,y_2)$ on $TW_i/TW_i\cap\Lambda\subset\check{Y}$, we have
\begin{eqnarray*}
v_i & = & \exp2\pi(x_1+\sqrt{-1}y_1),\\
u_i & = & \exp2\pi(-x_1-\sqrt{-1}y_1),\\
w & = & \exp2\pi(x_2+\sqrt{-1}y_2).
\end{eqnarray*}
Since $L_\gamma^\circ$ is a translate of the conormal bundle of $S_i=\{(x_1,x_2)\in W_i|x_2=0\}\subset W_i$, it follows from the constructions in Section \ref{SYZ_transform} that $C$ is defined by $w=\exp2\pi(0-\frac{\sqrt{-1}}{2})=-1$, or $1+w=0$. Hence, the closure $C$ of the submanifold mirror to $(L_\gamma^\circ,\nabla_0)$ is precisely the exceptional divisor $E_i\cong\bP^1$ in $\check{Y}$.

To compute the connection $\check{\nabla}$ mirror to $(L_\gamma^\circ,\nabla_0)$ and hence determine the holomorphic structure of $\mathscr{L}$, consider the Lagrangian submanifold $\tilde{L}_0\subset T^*W_i/T^*W_i\cap\Lambda^\vee$ defined by $x_2=0$ and $\xi_1=0$. The closure of the image $\Theta_i(\tilde{L}_0)$ is a Lagrangian sphere $L_{\gamma_0}$ which corresponds to a strongly admissible path $\gamma_0:[0,1]\to\bC^\times$ from $a_{i-1}$ to $a_i$. Fix a lift $\tilde{\gamma}_0:[0,1]\to\bR$ of the map $\gamma_0/|\gamma_0|:[0,1]\to S^1$. For any other admissible path $\gamma:[0,1]\to\bC^\times$ from $a_{i-1}$ to $a_i$ such that $\gamma(t)\not\in\Delta$ for $0<t<1$, there exists a unique lift $\tilde{\gamma}:[0,1]\to\bR$ of $\gamma/|\gamma|:[0,1]\to S^1$ such that $\tilde{\gamma}(0)=\tilde{\gamma}_0(0)$.
\begin{defn}
The winding number of $\gamma$ relative to $\gamma_0$ is the integer
\begin{equation*}
w(\gamma)=\tilde{\gamma}(1)-\tilde{\gamma}_0(1)\in\bZ.
\end{equation*}
\end{defn}
Note that the integer $w(\gamma)$ depends only on the isotopic class of $\gamma$. In fact, define a loop $\gamma\cup\bar\gamma_0:[0,1]\to\bC^\times$ by
\begin{equation*}
\gamma\cup\bar\gamma_0(t)=\left\{\begin{array}{ll}
                               \gamma(t) & \textrm{for $0\leq t\leq1$;}\\
                               \gamma_0(2-t) & \textrm{for $1\leq t\leq2$.}
                               \end{array}\right.
\end{equation*}
Then the class of $\gamma\cup\bar\gamma_0$ in $\pi_1(\bC^\times)\cong\bZ$ is precisely given by the integer $w(\gamma)$. See Figure \ref{rel_wind_num} below.

\begin{figure}[htp]
\begin{center}
\includegraphics[scale=0.8]{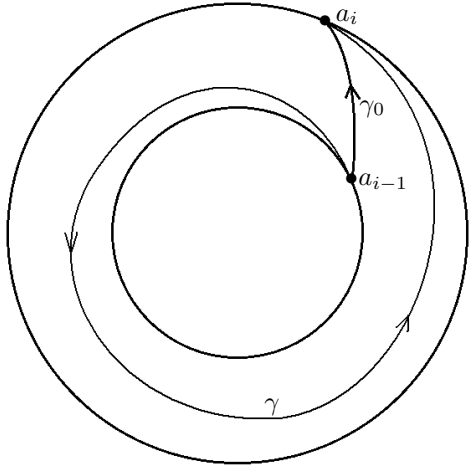}
\end{center}
\begin{picture}(100,-1)
\put(96,148){$\gamma_0$} \put(60,35){$\gamma$}
\put(95.5,120){$a_{i-1}$} \put(90.5,120){$\bullet$}
\put(87,182.5){$a_i$} \put(80.5,180){$\bullet$}
\end{picture}
\caption{The path $\gamma_0$ and a strongly admissible path $\gamma$ with winding number $w(\gamma)=1$ relative to $\gamma_0$ in the plane $\bC^\times$.}\label{rel_wind_num}
\end{figure}

Now the pullback of $L_\gamma$ under $\Theta_i$ can be written as $\tilde{L}_\gamma=\{(x_1(s),0,\xi_1(s),\xi_2):s\in(s_{i-1},s_i)\}$ for some functions $x_1(s)$ and $\xi_1(s)$ of $s\in(s_{i-1},s_i)$. Since (a lift of) the function $\xi_1(s)$ is bounded, we can extend it to $[s_{i-1},s_i]$.
\begin{lem}
$\xi_1(s_i)-\xi_1(s_{i-1})=w(\gamma)$.
\end{lem}
\begin{proof}
Since $w(\gamma)$ depends only on the isotopic class of $\gamma$, we can deform $\gamma$ to $\gamma'$ which is the concatenation of $\gamma_0$ with a loop winding around the image of the circle $\gamma_{s_i}$ for, say, $k$ times. Then $w(\gamma)=k$.

Restricting the symplectomorphism $\Theta_i$ to the preimage of $S_i:=(s_{i-1},s_i)\times\{0\}\subset W_i$. Under this restriction, deforming $\gamma$ corresponds to deforming the graph of $\xi_1(s)$ fixing the endpoints $\xi_1(s_{i-1})$ and $\xi(s_i)$. Along $\gamma_0$, the coordinate $\xi_1(s)$ is always equal to $\xi_1(s_{i-1})$ (which is equal to 0). Along the loop winding around $\gamma_{s_i}$, the difference $\xi_1(s_i)-\xi_1(s_{i-1})$ is precisely $k$ because winding once around the circle $\gamma_s$ corresponds to increasing $\xi_1(s)$ by one in the Lagrangian torus fiber $T_{s,0}$.

The result follows.
\end{proof}
Now the connection $\check{\nabla}$ mirror to $(L_\gamma^\circ,\nabla_0)$ can be extended over $C$ and is given by
\begin{equation*}
\check{\nabla}=d+2\pi\sqrt{-1}\xi_1(s)dy_1
\end{equation*}
with curvature two-form
\begin{equation*}
F=2\pi\sqrt{-1}d\xi_1(s)\wedge dy_1.
\end{equation*}
Hence, the degree of the holomorphic line bundle $\mathscr{L}\to C$ defined by $\check{\nabla}$ is
\begin{eqnarray*}
c_1(\mathscr{L})=\int_C\frac{\sqrt{-1}}{2\pi}F=-\int_{s=s_{i-1}}^{s=s_i}d\xi_1(s)=-(\xi_1(s_i)-\xi_1(s_{i-1}))=-w(\gamma).
\end{eqnarray*}
Summarizing these results, we have the
\begin{thm}\label{main_thm}
Given a strongly admissible path $\gamma:[0,1]\to\bC^\times$ from $a_{i-1}$ to $a_i$, the SYZ transformation of the A-brane $(L_\gamma,\nabla_0)$ on $Y$ is the B-brane $(C,\mathscr{L})$ on $\check{Y}$ where $C\cong\bP^1$ is the exceptional divisor $E_i$ and $\mathscr{L}=\mathcal{O}_{E_i}(-w(\gamma))$, i.e. we have
\begin{equation*}
\mathcal{F}(L_\gamma,\nabla_0)=(E_i,\mathcal{O}_{E_i}(-w(\gamma))).
\end{equation*}
\end{thm}

For each $i=1,\ldots,n$, we choose a strongly admissible path $\gamma_i:[0,1]\to\bC^\times$ going from $a_{i-1}$ to $a_i$ such that $w(\gamma_i)=1$. Denote by $L_i$ the Lagrangian sphere $L_{\gamma_i}\subset Y$. By the above theorem, the SYZ transformation of $(L_i,\nabla_0)$ is given by $\mathcal{E}_i:=\mathcal(L_i,\nabla_0)=(E_i,\mathcal{O}_{E_i}(-1))$.

We have
\begin{equation*}
|L_i\cap L_j|=\left\{\begin{array}{ll}
                        2 & \textrm{if $i=j$},\\
                        1 & \textrm{if $j=i+1$ or $j=i-1$},\\
                        0 & \textrm{otherwise}
                     \end{array}\right.
\end{equation*}
Equip $Y$ with the grading given by the 2nd tensor power of the holomorphic volume form
\begin{equation*}
\Omega=\frac{dv\wedge dz}{vz}.
\end{equation*}
Then there exists a suitable grading (in the sense of Kontsevich \cite{K94} and Seidel \cite{S00}) for each $L_i$ so that the $n$ graded Lagrangian spheres $\{\tilde{L}_i\}_{i=1}^n$ in $Y$ form an $A_n$-configuration of spherical objects \cite{S99,ST01} in the derived Fukaya category $D^b\textrm{Fuk}_0(Y,\omega)$ generated by $\{\tilde{L}_i\}_{i=1}^n$. This means that the Floer cohomology groups are given by
\begin{equation*}
\textrm{hom}_{\textrm{Fuk}_0(Y,\omega)}(\tilde{L}_i,\tilde{L}_j)=\left\{\begin{array}{ll}
                                        \bC\cdot r_i\oplus\bC\cdot s_i & \textrm{if $i=j$},\\
                                        \bC\cdot p_i & \textrm{if $j=i+1$},\\
                                        \bC\cdot q_i & \textrm{if $j=i-1$},\\
                                        0 & \textrm{otherwise},
                                     \end{array}\right.
\end{equation*}
where $r_i,s_i,p_i,q_i$ are intersection points such that $\textrm{deg}(r_i)=0$, $\textrm{deg}(s_i)=2$ and $\textrm{deg}(p_i)=\textrm{deg}(q_i)=1$.

Moreover, by the maximum principle, the Lagrangian submanifold does not bound any nonconstant holomorphic disks in $Y$. Hence, the $A_\infty$-operations are all trivial except for $\mathfrak{m}_2$, which can also be computed explicitly. See \cite[Section 5.3]{IUU10} for more details.

Now, let $E=E_1\cup E_2\cup\ldots\cup E_n$ be the union of exceptional divisors in $\check{Y}$. Consider the bounded derived category $D^b_E(\check{Y})$ of coherent sheaves on $\check{Y}$ supported at $E$. Let $D^b_0(\check{Y})$ be the full triangulated subcategory of $D^b_E(\check{Y})$ consisting of objects $\mathscr{E}$ such that $\bR g_*\mathscr{E}=0$. Then $D^b_0(\check{Y})$ is generated by $\mathcal{E}_1,\ldots,\mathcal{E}_n$. These generators also form an $A_n$-configuration of spherical objects in $D^b_0(\check{Y})$ \cite{ST01}.

Combining with the above discussion, we have the equality between morphism spaces:
\begin{equation*}
\textrm{hom}_{\textrm{Fuk}_0(Y,\omega)}(\tilde{L}_i,\tilde{L}_j)\cong\textrm{hom}_{D^b_0(\check{Y})}(\mathcal{E}_i,\mathcal{E}_j).
\end{equation*}
We conclude that
\begin{cor}\label{main_cor}
The SYZ transformation $\mathcal{F}$ induces an equivalence of triangulated categories
\begin{equation*}
\mathcal{F}:D^b\textrm{Fuk}_0(Y,\omega)\overset{\cong}{\longrightarrow}D^b_0(\check{Y}).
\end{equation*}
\end{cor}
\begin{proof}
This follows from the intrinsic formality of Seidel-Thomas \cite[Lemma 4.21]{ST01} as in the proofs of Theorem 28 and Lemma 40 in Ishii-Ueda-Uehara \cite{IUU10}.
\end{proof}
\begin{nb}
We have defined using SYZ transformations the functor $\mathcal{F}$ on objects only, and the proof that this gives an equivalence of triangulated categories is by computing morphisms on both sides. It would be desirable if we can use SYZ transformations to define the functor on morphisms as well. We plan to address this in the future.
\end{nb}
\begin{nb}
In \cite[Section 1.3]{ST01}, Seidel and Thomas mentioned that the braid group $B_{n+1}$-action on the derived Fukaya category of a symplectic manifold $Y$ defined by an $A_n$-configuration of Lagrangian spheres $L_1,\ldots,L_n$ is mirror to the $B_{n+1}$-action on the derived category of coherent sheaves of the mirror manifold $\check{Y}$ defined by the $A_n$-configuration of spherical objects $\mathcal{E}_1,\ldots,\mathcal{E}_n$ which are mirror to $L_1,\ldots,L_n$. To prove such a statement, we need to find a functor implementing the HMS equivalence $D^b\textrm{Fuk}(Y)\overset{\cong}{\rightarrow} D^b(\check{Y})$ which is equivariant with respect to the braid group actions on both sides. We conjecture that the geometric functor $\mathcal{F}:D^b\textrm{Fuk}_0(Y,\omega)\overset{\cong}{\rightarrow}D^b_0(\check{Y})$ we construct here satisfies this property.
\end{nb}

\section{Higher dimensional cases}\label{Higher_dim}

Most of the constructions in Sections \ref{Mirror} and \ref{HMS} can be generalized to higher dimensions. But they are not as explicit as in the 2-dimensional case. In this section, we give a brief sketch, leaving the details to a future paper.

Let $N=\bZ^n$ and $N'=\bZ^{n-1}\times\{0\}\subset N$. Let $P$ be a lattice polytope lying inside the hyperplane $N'_\bR\times\{1\}\subset N_\bR$. Let $\sigma$ be the cone over $P$ with vertex at the origin $0\in N_\bR$. Then the affine toric variety $X_\sigma$ associated to $\sigma$ is an isolated Gorenstein toric singularity. A resolution of singularities of $X_\sigma$ can be constructed by choosing a triangulation $\mathscr{P}$ of $P$. The fan $\Sigma$ consisting of cones over cells in $\mathscr{P}$ defines a smooth toric variety $X_\Sigma$ which yields a crepant resolution $X_\Sigma\to X_\sigma$. $X_\Sigma$ is a toric Calabi-Yau $n$-fold, meaning that the canonical line bundle $K_{X_\Sigma}$ is trivial. A 3-dimensional example is given by $X_\Sigma=K_{\bP^2}$.

Let $w_0,w_1,\ldots,w_m\in N'$ be the vertices of $\mathscr{P}$. Without loss of generality, we can assume that $w_0=0$. Then the primitive generators of the rays of $\Sigma$ are given by
\begin{equation*}
v_i:=(w_i,1)\in N,\ i=0,1,\ldots,m.
\end{equation*}
Consider the affine variety
\begin{equation*}
Y=\{(u,v,z_1,\ldots,z_{n-1})\in\bC^2\times(\bC^\times)^{n-1}|uv=f(z_1,\ldots,z_{n-1})\},
\end{equation*}
where
\begin{equation*}
f(z_1,\ldots,z_{n-1})=\sum_{i=0}^m a_iz_1^{w_i^1}\ldots z_{n-1}^{w_i^{n-1}}
\end{equation*}
is a Laurent polynomial on $(\bC^\times)^{n-1}$. This is a noncompact Calabi-Yau $n$-fold (again meaning that $K_Y$ is trivial). We equip $Y$ with the K\"ahler structure
\begin{equation*}
\omega=-\frac{\sqrt{-1}}{2}\Bigg(du\wedge d\bar{u}+dv\wedge d\bar{v}+\sum_{j=1}^{n-1}\frac{dz_j\wedge d\bar{z}_j}{|z_j|^2}\Bigg)\Bigg|_Y.
\end{equation*}

Similar to Section \ref{Mirror}, we can construct a Lagrangian torus fibration on $(Y,\omega)$ using the method of \cite{Go01,G01}. So we consider the Hamiltonian $S^1$-action
\begin{equation*}
e^{\sqrt{-1}\theta}\cdot(u,v,z_1,\ldots,z_{n-1})=(e^{\sqrt{-1}\theta}u,e^{-\sqrt{-1}\theta}v,z_1,\ldots,z_{n-1})
\end{equation*}
on $Y$ whose moment map is given by
\begin{equation*}
\mu:Y\to\bR,\ (u,v,z_1,\ldots,z_{n-1})\mapsto\frac{1}{2}(|u|^2-|v|^2).
\end{equation*}
Fix $\lambda\in\bR$, the reduced symplectic structure on $Y_\lambda:=\mu^{-1}(\lambda)/S^1$ can be explicitly computed to be
\begin{equation*}
\omega_\lambda=-\frac{\sqrt{-1}}{2}\Bigg(\frac{df\wedge d\bar{f}}{2\sqrt{\lambda^2+|f|^2}}+\sum_{j=1}^{n-1}\frac{dz_j\wedge d\bar{z}_j}{|z_j|^2}\Bigg).
\end{equation*}

Observe that the symplectic manifold $(Y_\lambda,\omega_\lambda)$ is symplectomorphic to $((\bC^\times)^{n-1},\omega_0)$ where $\omega_0$ is the standard symplectic form
\begin{equation*}
\omega_0=\sum_{j=1}^{n-1}\frac{dz_j\wedge d\bar{z}_j}{|z_j|^2}
\end{equation*}
on $(\bC^\times)^{n-1}$. Let $\Phi_\lambda:Y_\lambda\to(\bC^\times)^{n-1}$ be a symplectomorphism and let $\pi_\lambda:\mu^{-1}(\lambda)\to Y_\lambda$ be the quotient map. Also denote by $\textrm{Log}$ the map
\begin{equation*}
\textrm{Log}:(\bC^\times)^{n-1}\to\bR^{n-1},\ (z_1,\ldots,z_{n-1})\mapsto(\log|z_1|,\ldots,\log|z_{n-1}|).
\end{equation*}
Then the map $\rho:Y\to B:=\bR^n$ defined by
\begin{equation*}
\rho(u,v,z_1,\ldots,z_{n-1})=(\textrm{Log}\circ\Phi_\lambda\circ\pi_\lambda)(u,v,z_1,\ldots,z_{n-1})
\end{equation*}
for $(u,v,z_1,\ldots,z_{n-1})\in\mu^{-1}(\lambda)$ is a Lagrangian torus fibration. The appearance of $\Phi_\lambda$ here makes the mirror construction much less explicit.

We proceed with a description of the discriminant locus of $\rho$. First of all, as in the 2-dimensional case, the singularities of $\rho$ is given by the fixed points of the $S^1$-action:
\begin{equation*}
\Delta:=\{(0,0,z_1,\ldots,z_{n-1})\in\bC^2\times(\bC^\times)^{n-1}|f(z_1,\ldots,z_{n-1})=0\}.
\end{equation*}
Its image under $\rho$, which is the discriminant locus $\Gamma$ of $\rho$, is an \emph{amoeba-shaped} subset
\begin{equation*}
\Gamma=\mathscr{A}\times\{0\}\subset B
\end{equation*}
lying inside the hyperplane $\bR^{n-1}\times\{0\}$ in $B$. Notice that $\Gamma$ is of real codimension one in $B$; this is in sharp contrast with the 2-dimensional situation. The locus over which the Lagrangian torus fibers bound holomorphic disks is given by
\begin{equation*}
H:=\mathscr{A}\times\bR\subset B.
\end{equation*}
This is an open set (real codimension zero) in $B$, but we will continue to call it the wall in $B$. Notice that the components of the wall $H$ are all invariant under vertical translations.

To construct the mirror, we follow the SYZ proposal. So we first take the fiberwise dual fibration over $B^{sm}:=B\setminus\Gamma$. This gives the semi-flat mirror
\begin{equation*}
\check{Y}_0=TB^{sm}/\Lambda,
\end{equation*}
whose complex structure is not globally defined since the affine structure on $B$ is singular. We need to modify the gluing of complex charts by incorporating instanton corrections. Now, by considering the projection map
\begin{equation*}
\pi:Y\to(\bC^\times)^{n-1},\ (u,v,z_1,\ldots,z_{n-1})\mapsto(z_1,\ldots,z_{n-1})
\end{equation*}
and the maximum principle, we see that a Lagrangian torus fiber of $\rho$ bounds a nonconstant holomorphic disk if and only if it is a fiber over a point in $H$. Any such fiber bounds one (family of) holomorphic disks in $Y$. Therefore, the corrected gluing formula looks the same as in the 2-dimensional case. The combinatorics of the triangulation $\mathscr{P}$ of $P$ then determines the geometry of the instanton-corrected mirror $\check{Y}$. The result is given by the complement $\check{Y}$ of the hypersurface $D=\{\chi^{(0,1)=1}\}$ in the toric Calabi-Yau $n$-fold $X_\Sigma$. See Auroux \cite[Section 3.3]{A09} and the recent work Abouzaid-Auroux-Katzarkov \cite{AAL11} for more details.

The projection map $\pi:Y\to(\bC^\times)^{n-1}$ is a conic bundle with singular fibers over
\begin{equation*}
\Delta:=\{(z_1,\ldots,z_{n-1})\in(\bC^\times)^2|f(z_1,\ldots,z_{n-1})=0\}.
\end{equation*}
Consider an embedding
\begin{equation*}
\gamma:B^{n-1}\to(\bC^\times)^{n-1}
\end{equation*}
such that $\gamma(\partial B^{n-1})\subset\Delta$ and $\gamma(B^{n-1}\setminus\partial B^{n-1})\cap\Delta=\emptyset$, where $B^{n-1}$ is the closed unit ball in $\bR^{n-1}$. Let
\begin{eqnarray*}
L_\gamma & := & \{(u,v,z_1,\ldots,z_{n-1})\in Y|\\
&  & \qquad\qquad|u|=|v|,\ (\Phi_0\circ\pi_0)(z_1,\ldots,z_{n-1})\in\gamma(B^{n-1})\}.
\end{eqnarray*}
Then $L_\gamma$ is a Lagrangian in $(Y,\omega)$.

If the image of $L_\gamma$ under $\rho$ is the closure of a bounded component $S$ in the complement of the discriminant locus $\Gamma=\mathscr{A}\times\{0\}$ in the hyperplane $\bR^{n-1}\times\{0\}\subset B$, then $L_\gamma$ is diffeomorphic to a sphere $S^n$. In this case, the restriction of $\rho$ to the closure $\bar{S}$ of $S$ in $\bR^{n-1}\times\{0\}$ gives a circle fibration
\begin{equation*}
\rho_{L_\gamma}:L_\gamma\to\bar{S}
\end{equation*}
which collapses along the boundary $\partial S$. Let $L_\gamma^\circ$ be the preimage of $S$ under $\rho$. Then $L_\gamma^\circ$ is of the form we considered in Section \ref{SYZ_transform}, i.e. it is a translate of the conormal bundle of $S\subset B$. Equip $L_\gamma$ with a flat $U(1)$-connection $\nabla$. We get an A-brane $(L_\gamma,\nabla)$.

The SYZ transformation $\mathcal{F}(L_\gamma,\nabla)$ will be given by a B-brane $(C,\check{\nabla})$ where $C$ is a compact toric divisor in $\check{Y}=X_\Sigma\setminus D$ and $\check{\nabla}$ is a $U(1)$-connection which defines a holomorphic line bundle $\mathscr{L}$ over $C$. More precisely, if $(x_1,\ldots,x_{n-1},x_n)$ are the affine coordinates on $S$ and $(\xi_1,\ldots,\xi_n)$ are the fiber coordinates on $T^*S$, then $L_{\gamma}$ is defined by $x_n=\lambda=0$ and $\xi_j=\xi_j(x_1,\ldots,x_{n-1})$ for $j=1,\ldots,n-1$. Here, $\xi_j(x_1,\ldots,x_{n-1})$ are $C^\infty$ functions satisfying
\begin{equation*}
\frac{\partial\xi_j}{\partial x_l}=\frac{\partial\xi_l}{\partial x_j},\ j,l=1,\ldots,n-1.
\end{equation*}
By our definition in Section \ref{SYZ_transform}, the connection $\check{\nabla}$ is then given by
\begin{equation*}
\check{\nabla}=d+2\pi\sqrt{-1}\sum_{j=1}^{n-1}\xi_j(x_1,\ldots,x_{n-1})dy_j,
\end{equation*}
where $(y_1,\ldots,y_n)$ are the fiber coordinates on $TS$. The first Chern class of $\mathscr{L}$ is the class represented by
\begin{equation*}
\frac{\sqrt{-1}F_{\check{\nabla}}}{2\pi}=-\sum_{j,l=1}^{n-1}\frac{\partial\xi_j}{\partial x_l}dx_l\wedge dy_j.
\end{equation*}
Similar to the 2-dimensional case, the holomorphic structure of $\mathscr{L}$ are determined by certain winding numbers of $\gamma$ relative to a reference $\gamma_0$. This defines a functor (on the object level)
\begin{equation*}
\mathcal{F}:D^b\textrm{Fuk}_0(Y,\omega)\to D^b(\check{Y})
\end{equation*}
from the derived Fukaya category generated by graded Lagrangian spheres of the form $\tilde{L}_{\gamma}$ to the derived category of coherent sheaves of the mirror $\check{Y}$, which we expect to be an embedding of triangulated categories.

\end{document}